\documentclass[a4paper,times]{amsart}

\usepackage{amsthm,amsmath,amssymb}

\newtheorem{theorem}{\rm\bf Theorem}[section]
\newtheorem{proposition}[theorem]{\rm\bf Proposition}
\newtheorem{lemma}[theorem]{\rm\bf Lemma}
\newtheorem{corollary}[theorem]{\rm\bf Corollary}

\newtheorem*{theorem*}{Theorem}
\newtheorem*{theorem 1}{\rm\bf Proposition 1}
\newtheorem*{theorem 2}{\rm\bf Proposition 2}

\theoremstyle{definition}
\newtheorem{definition}[theorem]{\rm\bf Definition}

\theoremstyle{remark}
\newtheorem{remark}[theorem]{\rm\bf Remark}

\def\half#1#2{\begin{matrix}\frac{#1}{#2}\end{matrix}}

\def\R#1{\mathbb{R}^{#1}}

\def\scal#1#2{\langle #1,\, #2 \rangle}

\def\field{k}

\begin{document}

\title{A correction of the decomposability result in a paper by Meyer-Neutsch}

\author{Vladimir G. Tkachev}
\address{Department of Mathematics, Link\"oping University\\ Link\"oping, 58183, Sweden, vladimir.tkatjev@liu.se}

\begin{abstract}
In this short note, it is shown that there is a gap in the proof of Theorem~11 in the paper of Meyer and Neutsch (J. of Algebra, 1993). We prove, nevertheless, that the statement of the theorem is true and fix the proof by using a certain extremal property of idempotents which has an independent interest.
\end{abstract}

\keywords{
Commutative nonassociative algebras; Griess algebra; Idempotents; Associative bilinear form; Metrised algebras
}
\subjclass[2000]{
Primary 17A99, 17C27; Secondary 20D08}

\maketitle

\section{Introduction}
In their 1993 paper \cite{MeyerNeutsch}, Meyer and Neutsch established the existence of a $48$-dimensional associative subalgebra in the Griess algebra $\frak{G}$ and conjectured that $48$ is the largest possible dimension of an associative subalgebra of $\frak{G}$. This conjecture has been proved by Miyamoto \cite{Miyamoto96}.

One of the key results claimed in \cite{MeyerNeutsch}, Theorem~11 asserts that an idempotent $c\in \frak{G}$ is indecomposable if and only if the Peirce eigenspace $\frak{G}_c(1)$ with the eigenvalue $1$ is at most one-dimensional. This result was used by Miyamoto (see Theorem~6.8 in \cite{Miyamoto96}) and in a later research on associative subalgebras of low-dimensional Majorana algebras, see for instance  p.~576 in \cite{Castillo} and section~3.1 in \cite{Matsuo01}.

Unfortunately, there is a  gap in the argument in the proof of Theorem~11  in \cite{MeyerNeutsch}. Recall the statement of the theorem.

\begin{theorem}[Theorem~11, in \cite{MeyerNeutsch}]
An idempotent $a\in \frak{G}$ is indecomposable if and only if the Peirce eigenspace $\frak{G}_a(1)=\{x\in \frak{G}: L_ax=x\}$ is at most one-dimensional.
\end{theorem}

\begin{remark}
Note also that $\frak{G}_a(1)$ is a subalgebra in $\frak{G}$, see \cite{MeyerNeutsch}.
\end{remark}

More precisely, in their proof of the implication ``$\dim \frak{G}_c(1)\ge 2$ $\Rightarrow$ $c$ is decomposable'',  Meyer and Neutsch uses an erroneous argument that a (normalized) cubic form
\begin{equation}\label{cubic1}
\varphi(x):=\frac{\scal{x}{x^2}}{\scal{x}{x}^{\frac32}}
\end{equation}
provides at least two linearly independent idempotents.
Let us explain what is a gap here. It is well known (and easy to see)  that  $x\ne 0$ is a stationary point of $\varphi(x)$ precisely when $x^2=\lambda x$ for some $\lambda\in \R{}$. If $\lambda\ne 0$ then $x$ gives rise to a (nonzero) idempotent by normalizing: indeed $c=x/\lambda$ implies $c^2=c$. In order to get two distinct (linearly independent) idempotents, Meyer and Neutsch suggest to consider a minimum and a maximum points of $\varphi$ which would provide us with two desired idempotents (see  p.~15 in \cite{MeyerNeutsch}). But, since $\phi$ is an \textit{odd function}, this may give us two anti-collinear, and therefore \textit{linearly dependent}, elements. In particular, it is clear from \eqref{cubic1} that $x$ is a local maximum of $\varphi$ if and only if $-x$ is a local minimum of $\varphi$.

\begin{remark}
We will point out that the variational argument used by the authors in the proof of Lemma~2 and Theorem~8 in \cite{MeyerNeutsch} causes no problems  because one uses there an \textit{even} function (based on a quartic form). This, together with nonzero lower estimate (76), allows one to guarantee at least two linear independent stationary points (corresponding to a local maximum and minimum).
\end{remark}

The variational argument used in \cite{MeyerNeutsch} would be easily fixed if one could prove that a general cubic form always has at least $2$ linearly independent  stationary points with nonzero gradient\footnote{Note that the vanishing of gradient implies rather a 2-nilpotent, i.e. $x^2=0$, than idempotent. Therefore one assumes that $\lambda\ne 0$.}. However, this property fails in general,  as a counterexample~ in Section~\ref{ex1} below shows.

Thus, the method suggested in \cite{MeyerNeutsch} does not work in general. Nevertheless, the claim of Theorem~11 holds true. More precisely, we have the following  result which contains Theorem~11 in \cite{MeyerNeutsch} as a corollary.

\begin{theorem}
  \label{th1}
Let $V$ be a commutative nonassociative algebra over  $\R{}$  with an associative inner product $\scal{}{}$. Suppose  $c\ne 0$ is an idempotent of $V$ and the Peirce eigenspace $V_c(1)$ is an subalgebra of $V$. Then $\dim V_c(1)=1$ if and only if $c$ is indecomposable.
\end{theorem}

The theorem easily follows from an auxiliary result (Lemma~\ref{lem:main} below) which has an independent interest. The proof of our results uses the notion of metrised algebra. In the next section we recall some standard definitions and introduce the so-called algebras of cubic forms.

\section{Algebras of cubic forms}
Let $V$ be a vector space over  a field $\field$ of characteristic not $2$ or $3$. A multiplication on $V$ is a bilinear map (denoted by juxtaposition) $$
V\times V\ni (x,y)\to xy\in V.
$$
A vector space with a multiplication is called an algebra over $k$.

A symmetric $k$-bilinear form $Q(x,y):V\times V\to \field$ is called nonsingular if $Q(x,y)=0$ for all $y\in V$ implies $x=0$ \cite{Knus}. In what follows, we use the standard inner product notation for $Q$:
$$
Q(x,y)=\scal{x}{y}
$$
and the squared norm notation $|x|^2=\scal{x}{x}$. Note that $|x|$ may  be not defined in general.

If $V$ is an algebra then the bilinear form $\scal{\cdot}{\cdot}$ is called \textit{associative} \cite{Schafer}, \cite{Okubo81a}  if
\begin{equation}\label{Qass}
\scal{xy}{z}=\scal{x}{yz}, \qquad \forall x,y,z\in V,
\end{equation}
An algebra carrying  a non-singular symmetric bilinear form is called \textit{metrised}, cf.  \cite{Bordemann}. The standard examples are Lie algebras with the Killing form and Jordan algebras with the generic trace bilinear form. The most striking corollary of the existence of an associative bilinear form on an commutative algebra is that the operator of left ($=$right) multiplication
$$
L_x:y\to xy,
$$
is  self-adjoint, i.e.
\begin{equation}\label{selff}
\scal{L_xy}{z}=\scal{y}{L_xyz}, \qquad \forall x,y,z\in V.
\end{equation}

A function $u:V\to \field$ is called a cubic form if its full linearization
$$
u(x,y,z)=u(x+y+z)-u(x+y)-u(x+z)-u(y+z)+u(x)+u(y)+u(z)
$$
is a trilinear form. Note that thus defined linearization is obviously symmetric and the cubic form is recovered by
\begin{equation}\label{16}
6u(x)=u(x,x,x).
\end{equation}
The vector space of all cubic forms on $V$ will be denoted by $\mathcal{C}_3(V)$.

The following is an immediate corollary of the defintions.

\begin{proposition}\label{pro:cub}
Given a vector space $V$ with a  non-singular symmetric bilinear form $\scal{\cdot}{\cdot}$, there exists a canonical bijection between $\mathcal{C}_3(V)$ and commutative metrised algebras $(V,\scal{\cdot}{\cdot})$, where the multiplication  $(x,y)\to xy$ is uniquely determined by
\begin{equation}\label{muu}
xy :=\text{  the unique element satisfying } \scal{xy}{z}=u(x,y,z) \text{ for all } z\in V
\end{equation}
\end{proposition}

Given a cubic form $u$ on $V$, the commutative nonassociative algebra obtained according to Proposition~\ref{pro:cub} is denoted by $V(u)$ and is called the \textit{algebra of the cubic form} $u$.
Conversely, given a commutative metrised algebra $A=(V,\scal{\cdot}{\cdot})$, define the cubic form
\begin{equation}\label{recover}
u_A(x):=\frac16 \scal{xx}{x}.
\end{equation}
The factor $\frac16 $ is chosen to agree with \eqref{16}.
Then $A$ is  isomorphic to $V(u_A)$.

We point out that as a corollary of the symmetry of $u(x,y,z)$, the algebra $V(u)$ is always commutative but may be nonassociative.  It is easy to see that the following is true.

\begin{proposition}\label{pro:zeroalgebra}
$V(u)$ is a zero algebra if and only if $u\equiv 0$.
\end{proposition}


Now we consider the specialization of the above definitions  when $V$ is a vector space over $k=\R{}$ or $\mathbb{C}$. Let $\scal{\cdot}{\cdot}$ be some non-singular symmetric bilinear form on $V$ and let us fix a cubic form $u$. Then it readily follows from \eqref{recover} and \eqref{muu} that the gradient of $u(x)$ is essentially the square of the element $x$ (in $V(u)$):
\begin{equation}\label{grad}
Du(x)=\frac{1}{2}xx=\frac{1}{2}x^2,
\end{equation}
and, furthermore, the multiplication in $V(u)$ is explicitly recovered by
\begin{equation}\label{hess}
(D^2u(x))y=xy,
\end{equation}
where $D^2u(x)$ is the Hessian matrix of $u$ at $x$. This observation is an important ingredient in the proof of the main results of \cite{MeyerNeutsch}, in particular of  Theorem~11. The relations \eqref{grad} and \eqref{hess} are also important tools in  recent applications of nonassociative algebras to global analysis, see \cite[Ch.~6]{NTVbook}, \cite{Tk15b}. Note also that if $k$ is an arbitrary field, then \eqref{grad} and \eqref{hess} can be interpreted as formal definitions of the gradient and the Hessian of $u$.


The  case $\field =\R{}$ has a particular interest as the following property shows.

\begin{lemma}\label{lem:main}
Let $(V,\scal{\cdot}{\cdot})$ be a (nonzero) commutative metrized algebra over $\R{}$, $\scal{\cdot}{\cdot}$ being positive definite. The set $E$ of constrained stationary points of the variational problem
\begin{equation}\label{variational}
\scal{x}{x^2}\to \max\quad \text{ subject to a constraint}\quad \scal{x}{x}=1
\end{equation}
is nonempty and the maximum is attained. Denote by $E_0\subset E$ the (nonempty) set of local maxima in \eqref{variational}. Then for any $x\in E$, either $x^2=0$ or $c:=x/\scal{x^2}{x}$ is an idempotent in $V$. If additionally $x\in E_0$ then $\scal{x^2}{x}>0$ and the corresponding idempotent  $c=x/\scal{x^2}{x}$ satisfies the extremal property
\begin{equation}\label{extremal}
L_c\le \half12 \text{ on }c^\bot:=\{x\in V:\scal{x}{c}=0\}.\end{equation}
In particular, the eigenvalue $1$ of $L_c $ is simple.
\end{lemma}

\begin{proof}
By the positive definiteness of $\scal{\cdot}{\cdot}$, the unit sphere $S=\{x\in V:\scal{x}{x}=1\}$  is compact in the standard Euclidean topology on $V$ induced by $|x|^2=\scal{x}{x}$.
By the assumption and  Proposition~\ref{pro:zeroalgebra}, the cubic form $u(x)=\half16\scal{x^2}{x}\not\equiv0$. Note that the algebra $V(u)$ is isomorphic to $V$.

Since $u$ is continuous as a function on $S$, it attains its maximum value there. In particular both $E_0$ and $E$ are nonempty. Let $x\in E$ be a stationary point, then
using Lagrange’s undetermined multipliers and \eqref{grad} we obtain for the directional derivative
$$
0=\partial_y u|_{x}=\half{1}{2}\scal{x^2}{y}\quad
\text{for any $y\in V$ and $\scal{x}{y}=0$}.
$$
which implies that $x$ and $x^2$ are parallel, i.e. $x^2=\lambda x$ where $\lambda=\scal{x^2}{x}$. Thus, $x^2=0$ if and only if $\scal{x^2}{x}=0$, otherwise scaling appropriately we have that $c:=x/\scal{x^2}{x}$ is  a nonzero idempotent in $V$.

To prove the second claim of the lemma, note that the function $f(x)=\frac{\scal{x^2}{x}}{|x|^3}$ is a homogeneous of degree zero and smooth outside the origin. We have
\begin{equation}\label{D1f}
\partial_y f|_{x}=\frac{3(\scal{x^2}{y}|x|^2-\scal{x^2}{x}\scal{x}{y})}{|x|^5},
\end{equation}
implying by duality the expression for the gradient
$$
\half13Df(x)=\frac{x^2|x|^2-\scal{x^2}{x}x}{|x|^5}.
$$
Arguing similarly, one obtains
$$
\half13\partial_z Df(x)=\frac{2|x|^4xz-3|x|^2(x^2\scal{x}{z}+\scal{x^2}{z}x) -\scal{x^2}{x}|x|^2z +5\scal{x}{z}\scal{x^2}{x}x}{|x|^7},
$$
hence the Hessian  is given by
\begin{equation}\label{D2f}
\half13D^2f(x)=
\frac{2|x|^4L_x-\scal{x^2}{x}|x|^2-3|x|^2(x\otimes x^2+x^2\otimes x)+5\scal{x^2}{x} x\otimes x}{|x|^5}.
\end{equation}
Now let $x\in E_0$. Then  $x$ is a local maximum of $f$ and obviously $f(x)=\scal{x^2}{x}>0$ (because $\scal{x^2}{x}\not\equiv0$ on $V$). By the above $c=x/\scal{x}{x^2}$ is an idempotent, and by the homogeneity $f(x)=f(c)$, hence $c$ is a local maximum of $f$ too. Therefore the maximum condition implies that the Hessian matrix is non-positive definite
$$
0\ge \half13D^2f(c)=
\frac{(2L_c-1)|c|^2\,-c\otimes c}{|c|^3}.
$$
Since $c\otimes c=0$ on $c^\bot$ we conclude that $2L_c-1\le 0$ on $c^\bot$, thus, the desired conclusion follows.

\end{proof}

\begin{definition}
An idempotent $c\in (V,Q)$ corresponding to any local maximum in \eqref{variational} is said to be \textit{extremal}.
\end{definition}

An important corollary of Lemma~\ref{lem:main} is the following

\begin{corollary}\label{cor:1}
Let a nonzero \textit{unital} commutative algebra $V$ over $\R{}$ possess a symmetric positive definite associative bilinear form and $\dim V\ge 2$. Then there exists an idempotent in $V$ distinct from the unit. In particular, the unit is decomposable and $V$ has at least three distinct nonzero idempotents.
\end{corollary}

\begin{proof}
By Lemma~\ref{lem:main} there exists a (nonzero) extremal idempotent $c\in V$. By the extremal property, one also has $L_c\le \half12$ on $c^\bot$. Note that the subspace $c^\bot$ is nontrivial because $\dim c^\bot=\dim V-1\ge 1$. Let $e\in V$ be the algebra unit. Then $e$ is an idempotent and $e\ne c$ because $L_e\equiv 1$ on $V$. Then $(e-c)^2=e-2ec+c=e-c$, hence $e-c$ is also an idempotent. The claim follows.
\end{proof}

Now we are able to finish the proof of the main result.
\begin{proof}[Proof of Theorem~\ref{th1}]
First recall that an idempotent is called decomposable if it may be expressed as a sum of at least two non-zero idempotents; otherwise, the idempotent is called indecomposable. Now, if $c$ is decomposable, say $c=c_1+c_2$, $c_i$ being nonzero idempotents, then $c_1c_2=0$ and $c_1,c_2$ are obviously  linearly independent. One has therefore $cc_i=c_i$, hence $c_i\in V_c(1)$, implying $\dim V_c(1)\ge2$.

In the converse direction, suppose that $c\ne0$ be a idempotent in $V$ with $\dim V_c(1)\ge 2$ and $V_c(1)$ is a subalgebra of $V$. Note that $c$ acts on the subalgebra $V_c(1)$ as a unit. Therefore by Corollary~\ref{cor:1}, $c$ is decomposable.

\end{proof}

\section{A counterexample}
\label{ex1}

Here we construct an example of a cubic form whose stationary points  with nonzero gradient lies on a line. On the nonassociative algebra level, this shows that  for any $n\ge 2$ there exist nontrivial $n$-dimensional algebras over $\R{}$ with associative bilinear form having only one nonzero idempotent.

Let $n\ge 2$ and consider a cubic form $u(x)=\frac13x_1(x_1^2+3a_2x_2^2+\ldots +3a_nx_n^2)$ in  $\R{n}$ equipped with the standard Euclidean scalar product. Suppose that all $a_i$ are different and $0<a_i<\frac12$. We claim that any stationary point of $\varphi(x)=u(x)/|x|^3$ \textit{with nonzero gradient} are parallel to $\xi:=(1,0,\ldots,0)$.
Indeed, given a stationary point $x$  of $\varphi(x)$, the stationary point condition readily yields $Du(x)=\lambda x$ for some real $\lambda$. By the assumption $Du(x)\ne 0$, hence $\lambda\ne 0$, in particular $x\ne 0$. Since $u(-x)=-u(x)$, one has $Du(x)=Du(-x)$, thus interchanging $x$ and $-x$, one may assume without loss of generality that $\lambda>0$, and, also by the homogeneity of $u$ that $\lambda=1$. This yields
\begin{align}
x_1^2+(a_2x_2^2+\ldots +a_nx_n^2)&=x_1,\label{e1}\\
2a_kx_kx_1&=x_k, \quad\text{for }k=2,\ldots, n.\label{e2}
\end{align}
If $x_k=0$ for $2\le k\le n$ then by \eqref{e1}  $x_1^2=x_1$, thus $x_1=1$, i.e. $x=\xi$. If at least two $x_i$ and $x_j$ are nonzero for $2\le i,j\le n$ then \eqref{e2} implies $x_1=\frac1{2a_i}=\frac1{2a_j}$, a contradiction with $a_i\ne a_j$.
Therefore there can exist only one nonzero $x_i\ne 0$ for some $2\le i\le n$. But in that case, $x_1=\frac1{2a_i}$ and \eqref{e1} yields by the assumptions on $a_i$ that $x_i^2=\frac{2a_i-1}{4a_i^3}<0$, a contradiction again. This proves the claim.

\begin{remark}
Some final remarks are also in order. In the proof of Lemma~\ref{lem:main}, we made an essentially use of the condition that the associative form is positive definite. It is unclear, however, if Theorem~\ref{th1} remains true if one drops the definiteness condition and consider a weaker assumption, for example, that $\scal{\cdot}{\cdot}$  is merely non-singular.
\end{remark}


\bibliographystyle{plain}

\def\cprime{$'$} \def\cprime{$'$}

\end{document}